\documentclass[12pt]{amsart}
\usepackage[utf8]{inputenc}
\usepackage[T1]{fontenc}
\usepackage{amsmath, amssymb, amsthm}
\usepackage{geometry}
\usepackage{booktabs}
\usepackage{array}
\usepackage{graphicx}
\usepackage{hyperref}
\usepackage{multicol}
\usepackage{enumitem}
\usepackage{tikz}
\usepackage{tikz-3dplot}
\newtheorem{theorem}{Theorem}[section]
\newtheorem{remark}{Remark}[section]
\theoremstyle{definition}
\newtheorem{definition}{Definition}[section]
\newtheorem{lemma}{Lemma}[section]
\newtheorem{proposition}{Proposition}[section]
\newtheorem{corollary}{Corollary}[section]
\newtheorem{example}{Example}[section]

\title[On Some Weighted Sine Inequalities]{On some inequalities for weighted products and ratios of the sine function}

\author {Augustine L. Mahu}
\address{Department of Mathematics, University of Ghana, PO. Box LG 62 Legon, Accra, Ghana }
\email{ almahu@ug.edu.gh}
\author {Beno\^it F. Sehba}
\address{Department of Mathematics, University of Ghana, PO. Box LG 62 Legon, Accra, Ghana }
\email{ bfsehba@ug.edu.gh}
\author {Cecilia D. Williams}
\address{Department of Mathematics, University of Ghana, PO. Box LG 62 Legon, Accra, Ghana }
\email{ cdwilliams@st.ug.edu.gh}

\begin{document}


\begin{abstract}
We introduce a reflection-substitution technique for sine inequalities that yields, via H\"older's inequality and its reverse, $2^{n-1}$ distinct upper bounds for products and lower bounds for ratios of weighted sines. The comparison between bounds reduces to simple sign conditions on linear combinations of the variables. We fully classify the two- and three-dimensional cases, providing explicit dominance criteria with detailed examples for each signature.
\end{abstract}

\maketitle


\thispagestyle{empty}


\section{Introduction}

Inequalities involving trigonometric functions are fundamental in analysis, geometry, and approximation theory. Among the most classical results is the sine-concavity inequality
\begin{equation}\label{eq:prodineqsine}
  \prod_{i=1}^{n}\sin^{\alpha_i}(\theta_i)
  \;\leq\;
  \sin^{\alpha}\!\left(\sum_{i=1}^n\frac{\alpha_i}{\alpha}\,\theta_i\right),
  \qquad
  \alpha=\sum_{i=1}^n\alpha_i,
\end{equation}
valid for $\theta_i\in(0,\pi)$ and positive weights $\alpha_i$. This follows directly from Jensen's inequality applied to the concave function $\log\sin$ on $(0,\pi)$ (see for example \cite{Bullen2003,Mitrinovic1970,MPF1993}).

A subtler direction is the derivation of sharp lower bounds for ratio expressions of the form
\begin{equation}\label{eq:ratioform}
  \frac{\displaystyle\prod_{i=1}^{k}\sin^{\alpha_i}(\pi x_i)}{\displaystyle\prod_{i=k+1}^{n}\sin^{\alpha_i}(\pi x_i)},
  \qquad x_i\in(0,1).
\end{equation}
Such ratios arise naturally when solving~\eqref{eq:prodineqsine} for one variable in terms of the others, and appear in integral estimates for beta and gamma functions, Euler-type integrals.

In a recent paper~\cite{MSW2026}, we revisited inequality~\eqref{eq:prodineqsine} via the log-convexity of the Gamma function together with Euler's reflection formula. This approach, after the substitution $\theta_i=\pi\lambda_i$, is equivalent to applying H\"older's inequality to the Beta-function integral. For specific paired or cyclic weight choices, we exhibited two competing upper bounds with explicit sign-condition criteria determining which bound is sharper.

The present paper generalizes this phenomenon significantly. We observe that the freedom to represent $\sin(\pi x_i)$ through either of the two Beta-integral representations attached to the reflection identity $\sin(\pi x_i)=\sin(\pi(1-x_i))$ is available independently at every coordinate $i=1,\dots,n$. Exploiting this freedom in full generality produces, for $n$ weights, not two but $2^{n-1}$ distinct upper bounds for the weighted product in~\eqref{eq:prodineqsine}. The bound of~\cite{MSW2026} reappears as one member of this larger family.

Dualizing the argument by means of the reverse H\"older inequality, we show that the same reflection-substitution technique produces $2^{n-1}$ distinct lower bounds for ratios of the form~\eqref{eq:ratioform}, a direction not previously addressed.

The comparison of the bounds within each family reduces to simple sign conditions on linear combinations of $(2x_i-1)$, so the sharpest bound can always be selected explicitly. We treat the two- and three-dimensional cases in complete detail; these settings arise most often in applications and already display the full range of sign configurations while remaining transparent enough to admit closed-form dominance criteria.

We are motivated by a simple observation that the classical proof of~\eqref{eq:prodineqsine} overlooks; the reflection identity $\sin(\pi x_i)=\sin(\pi(1-x_i))$ gives each angle two equally valid representations, yet Jensen's argument commits to only one of them, with no mechanism for asking whether a different choice might do better. Our earlier observation that this freedom, exercised at a single coordinate, already produces two competing bounds with a clean dominance criterion~\cite{MSW2026} suggested that the phenomenon was not accidental but structural and that restricting it to one coordinate was leaving most of the picture unexplored. This paper follows that freedom to its natural conclusion, letting it act independently at every coordinate.

The paper is organized as follows. Section~\ref{sec:prelim} recalls some necessary tools needed in our presentation. In Section~\ref{sec:method}, we present general results for upper bounds on products and lower bounds on ratios. In Section~\ref{sec:compare}, we compare the members of each family of bounds in two and three dimensions, identifying the exact dominance region for every sign pattern. Section~\ref{sec:examples} provides concrete examples illustrating the theory in all cases, including numerical computations. We end the paper with a conclusion section.


\section{Preliminaries}
\label{sec:prelim}

Recall that the Beta function is defined for $x>0$, $y>0$ by
\[
B(x,y):=\int_0^1 t^{x-1}(1-t)^{y-1}\,dt.
\]
For any $0<x<1$, the following reflection formula holds:
\begin{equation}\label{eq:reflection}
B(1-x,x)=\frac{\pi}{\sin(\pi x)}.
\end{equation}
This identity is classical (see for example \cite{Artin1964,BohrMollerup1922}). It is precisely the symmetry $\sin(\pi x)=\sin(\pi(1-x))$ hidden in~\eqref{eq:reflection} that will drive the reflection-substitution technique.

The following classical H\"older inequality is one of the key tools for finding upper bounds. For its proof, we refer the reader to \cite{HLP1934,Rudin1987}.

\begin{lemma}\label{lem:holder}
Let $f_1,\ldots,f_m$ be positive measurable functions on a measure space $(X,\mu)$.
Let $p_1,\ldots,p_m \in (1,\infty)$ satisfy $\sum_{i=1}^{m} 1/p_i = 1$. Then
\[
\int_X \prod_{i=1}^{m} f_i(x) \, d\mu(x) \leq \prod_{i=1}^{m} \left(\int_X f_i(x)^{p_i} \, d\mu(x)\right)^{1/p_i}.
\]
\end{lemma}

For our lower bounds, we use the following generalized reverse H\"older inequality (see \cite{HLP1934}).

\begin{lemma}\label{lem:reverseholder}
Let $f_1, \dots, f_m$ be positive measurable functions on a measure space $(X, \mu)$.
Let $p_1, \dots, p_m \in \mathbb{R} \setminus \{0\}$ satisfy $\displaystyle\sum_{i=1}^m \frac{1}{p_i} = 1$.
If at least one $p_i \in (0,1)$, then
\[
\int_X \prod_{i=1}^m f_i(x) \, d\mu(x) \;\ge\; \prod_{i=1}^m \left( \int_X f_i(x)^{p_i} \, d\mu(x) \right)^{1/p_i},
\]
where for $p_i < 0$, the term $f_i^{p_i} = 1 / f_i^{|p_i|}$, and all integrals are assumed finite.
\end{lemma}

The following elementary comparison lemma for the sine function will be used repeatedly.

\begin{lemma}\label{lem:sine-compare}
Let $\lambda > 0$ and $\sigma,\tau \in (0,\lambda)$. Then
\[
\sin\left(\frac{\pi\sigma}{\lambda}\right) \geq \sin\left(\frac{\pi\tau}{\lambda}\right) \iff (\sigma-\tau)(\sigma+\tau-\lambda) \leq 0,
\]
with equality if and only if $\sigma = \tau$ or $\sigma + \tau = \lambda$.
\end{lemma}

\begin{proof}
This was proved in \cite{MSW2026}. We give here another proof.

By the sum-to-product identity, we have
\[
\sin\left(\frac{\pi\sigma}{\lambda}\right) - \sin\left(\frac{\pi\tau}{\lambda}\right) = 2\cos\left(\frac{\pi(\sigma+\tau)}{2\lambda}\right)\sin\left(\frac{\pi(\sigma-\tau)}{2\lambda}\right).
\]
Since $\sigma,\tau \in (0,\lambda)$, we have $\operatorname{sgn}(\cos \frac{\pi(\sigma+\tau)}{2\lambda}) = \operatorname{sgn}(\lambda-\sigma-\tau)$ and $\operatorname{sgn}(\sin \frac{\pi(\sigma-\tau)}{2\lambda}) = \operatorname{sgn}(\sigma-\tau)$. Therefore $\sin(\pi\sigma/\lambda) - \sin(\pi\tau/\lambda) \geq 0$ if and only if $(\sigma-\tau)(\sigma+\tau-\lambda) \leq 0$. Equality holds precisely when $\sigma=\tau$ or $\sigma+\tau=\lambda$.
\end{proof}


\section{The reflection-substitution method and main general results}
\label{sec:method}
In this section, we give two general results on the upper bounds for the weighted products and lower bound for ratios of sine function. We start by explaining our approach and illustrate it with examples.
\subsection{Intuition and basic idea}

The reflection formula~\eqref{eq:reflection} says that $\pi/\sin(\pi x_i)$ is computed by the same Beta integral whether we present the exponent pair as $(x_i,1-x_i)$ or as $(1-x_i,x_i)$. Indeed, since $B(1-x_i,x_i)=B(x_i,1-x_i)$, the integrand $t^{x_i-1}(1-t)^{-x_i}$ may equally well be written as $t^{(1-x_i)-1}(1-t)^{-(1-x_i)}$ after exchanging $t$ and $1-t$.

In other words, each coordinate $x_i$ enters the Beta integral representation of $\pi/\sin(\pi x_i)$ through two equivalent, interchangeable exponent choices: the angle $x_i$ itself, or its complement $1-x_i$. Because this choice is made independently for every coordinate, a weighted product of $n$ sine values admits $2^n$ apparently different but pointwise-equal integral representations.

\begin{example}
For $n=2$, the four representations are
\[
\begin{array}{|c|c|}
 \hline \text{Choice} & \text{Integrand exponent pattern} \\
\hline
(1,1) & (x_1, x_2) \\
(1,0) & (x_1, 1-x_2) \\
(0,1) & (1-x_1, x_2) \\
(0,0) & (1-x_1, 1-x_2)\\ \hline
\end{array}
\]
where a $1$ means using the original angle and $0$ means using its complement.
\end{example}

Applying H\"older's inequality (or its reverse) to each of these representations produces, a priori, $2^n$ inequalities; in fact $2^{n-1}$ distinct inequalities.

\subsection{Upper bounds for weighted products}
\label{ssec:upper}

Let $\alpha_1,\ldots,\alpha_n > 0$ and $x_1,\ldots,x_n \in (0,1)$, and set
$\alpha := \sum_{i=1}^{n} \alpha_i > 0$. Write $\mathcal{N} = \{1,\ldots,n\}$.

\begin{definition}\label{def:upper-reflection}
A \emph{reflection set} is any subset $E \subseteq \mathcal N$ such that for
each $i\in\mathcal N$, the associated selected angle $a_i$ satisfies
\[
a_i \;=\; \begin{cases} x_i, & i \notin E,\\ 1-x_i, & i \in E ;\end{cases}
\]
and the associated weighted reflection mean is given by
\begin{equation}\label{eq:Mdef}
\begin{aligned}
M_E(\vec x) &\;=\; \frac{1}{\alpha}\sum_{i=1}^n \alpha_i a_i
= \frac{\displaystyle\sum_{i=1}^{n} \alpha_i x_i - \sum_{i \in E} \alpha_i (2x_i - 1)}{\alpha}.
\end{aligned}
\end{equation}
\end{definition}

\begin{example}
For $n=2$, $\alpha_1=\beta$, $\alpha_2=\gamma$, $x_1=x$, $x_2=y$. Take as reflection sets $E=\emptyset$ and $E=\{2\}$. For the set $\emptyset$, as $1\notin \emptyset$ and $2\notin\emptyset$, both $x=x_1$ and $y=x_2$ are kept and the reflection mean is
\[
M_\emptyset = \frac{\beta x + \gamma y}{\beta+\gamma}.\]
For the set $\{2\}$, we have that $1\notin\{2\}$ and $2\in\{2\}$, so $x=x_1$ is kept while $y=x_1$ is flipped to $1-y$. The reflection mean is then
\[
M_{\{2\}} = \frac{\beta x + \gamma(1-y)}{\beta+\gamma}.
\]
The other sets $E=\{1\}$ and $E=\{1,2\}$ give complementary means $1-M_\emptyset$ and $1-M_{\{2\}}$, which yield the same sine values.
\end{example}

Our main theorem on upper bounds is as follows.

\begin{theorem}\label{thm:upper-main}
Let $\alpha_1,\ldots,\alpha_n > 0$ and $x_1,\ldots,x_n \in (0,1)$. Let
$E \subseteq \mathcal N$ be any reflection set with $0 < M_E(\vec x) < 1$. Then
\begin{equation}\label{eq:upperbound}
\prod_{i=1}^{n} \sin^{\alpha_i}(\pi x_i) \;\leq\; \sin^{\alpha}(\pi M_E(\vec x)).
\end{equation}
\end{theorem}

\begin{proof}
Define the exponents $p_i = \alpha/\alpha_i$ for $i=1,\ldots,n$. Then $\sum_{i=1}^{n} 1/p_i = 1$,
and since $\alpha=\sum_j\alpha_j>\alpha_i$ for each $i$, every $p_i\in(1,\infty)$.

Put $M = M_E(\vec x)$ and $N = 1 - M$. Using the integral representation of the Beta function, we get
\[
B(M,N) = \int_{0}^{1} t^{M-1}(1-t)^{-M}\,dt
= \int_{0}^{1} \prod_{i=1}^{n} \bigl[t^{a_i}(1-t)^{1-a_i}\bigr]^{1/p_i} \cdot t^{-1}(1-t)^{-1}\,dt,
\]
where $a_i$ is the angle selected by $E$ as in Definition~\ref{def:upper-reflection}. Using~\eqref{eq:Mdef}, the exponent of $t$ sums to $M-1$ and the exponent of $(1-t)$ sums to $-M$.

Applying the H\"older's inequality, we get
\[
B(M,N) \leq \prod_{i=1}^{n} \left(\int_{0}^{1} t^{a_i-1}(1-t)^{-a_i}\,dt\right)^{1/p_i}.
\]
For any $i$, regardless of whether $a_i = x_i$ or $a_i = 1 - x_i$, we have
\[
\int_{0}^{1} t^{a_i-1}(1-t)^{-a_i}\,dt = \frac{\pi}{\sin(\pi x_i)}
\]
using the reflection formula and the symmetry $B(x,y)=B(y,x)$ of the Beta function. Hence,
\[
B(M,N) \leq \prod_{i=1}^{n} \left(\frac{\pi}{\sin(\pi x_i)}\right)^{\alpha_i/\alpha}.
\]
Raising both sides to the power $\alpha > 0$ and using $B(M,N) = \pi/\sin(\pi M)$ gives
\[
\left(\frac{\pi}{\sin(\pi M)}\right)^{\alpha} \leq \prod_{i=1}^{n} \left(\frac{\pi}{\sin(\pi x_i)}\right)^{\alpha_i}.
\]
As $\prod_{i=1}^{n} \pi^{\alpha_i} = \pi^{\alpha}$, this simplifies to~\eqref{eq:upperbound}. The proof is complete.
\end{proof}

\subsection{Lower bounds for weighted ratios}
\label{ssec:lower}

Let $\alpha_1,\ldots,\alpha_n > 0$ and $x_1,\ldots,x_n \in (0,1)$. Fix an integer
$1 \le k < n$ and write $\mathcal{I} = \{1,\ldots,k\}$ (numerator indices) and
$\mathcal{J} = \{k+1,\ldots,n\}$ (denominator indices). Set
\[
\alpha := \sum_{i\in\mathcal I} \alpha_i - \sum_{j\in\mathcal J} \alpha_j > 0,
\]
and assume there is at least one $i_0\in\mathcal I$ with $\alpha_{i_0}>\alpha$; this
guarantees that the exponent $p_{i_0}=\alpha/\alpha_{i_0}$ lies in $(0,1)$, as required to
invoke the reverse H\"older inequality.

\begin{definition}\label{def:lower-reflection}
A \emph{reflection pair} $(E,F)$ is a choice of subsets $E \subseteq \mathcal I$ and
$F\subseteq\mathcal J$ such that for each $i\in \mathcal I$ and $j\in\mathcal J$, the associated selected angles satisfy
\[
a_i = \begin{cases} x_i, & i\in\mathcal I\setminus E\\ 1-x_i, & i\in E\end{cases}
\qquad (i\in\mathcal I),
\]
\[
c_j = \begin{cases} x_j, & j\in\mathcal J\setminus F\\ 1-x_j, & j\in F\end{cases}
\qquad (j\in\mathcal J),
\]
and the associated weighted reflection mean is given by
\begin{equation}\label{eq:MEFdef}
M_{E,F}(\vec x) = \frac{\displaystyle\sum_{i\in\mathcal I}\alpha_i x_i - \sum_{i\in E}\alpha_i(2x_i-1)
- \sum_{j\in\mathcal J}\alpha_j x_j + \sum_{j\in F}\alpha_j(2x_j-1)}{\alpha}.
\end{equation}
\end{definition}

Our main result on lower bounds is as follows.

\begin{theorem}\label{thm:lower-main}
Let $\alpha_1,\ldots,\alpha_n,\alpha$ be as above, and let $(E,F)$ be any reflection pair
with $0<M_{E,F}(\vec x)<1$. Then
\begin{equation}\label{eq:genineq}
\sin^{\alpha}\!\left(\pi M_{E,F}(\vec x)\right)
\;\leq\;
\frac{\displaystyle\prod_{i\in\mathcal I} \sin^{\alpha_i}(\pi x_i)}
     {\displaystyle\prod_{j\in\mathcal J} \sin^{\alpha_j}(\pi x_j)}.
\end{equation}
\end{theorem}

\begin{proof}
Define exponents $p_i = \alpha/\alpha_i$ for $i\in\mathcal I$ and $q_j = -\alpha/\alpha_j$
for $j\in\mathcal J$. Then $\sum_{i\in\mathcal I} 1/p_i + \sum_{j\in\mathcal J} 1/q_j = 1$,
and $p_{i_0}\in(0,1)$.

Put $M = M_{E,F}(\vec x)$ and $N = 1-M$. Using the integral representation of the Beta
function, we obtain
\begin{align*}
B(M,N) &= \int_0^1 t^{M-1}(1-t)^{-M}\,dt \\
&= \int_0^1 \prod_{i\in\mathcal I} \bigl[ t^{a_i}(1-t)^{1-a_i} \bigr]^{1/p_i}
         \prod_{j\in\mathcal J} \bigl[ t^{c_j}(1-t)^{1-c_j} \bigr]^{1/q_j}
         \cdot t^{-1}(1-t)^{-1} \, dt,
\end{align*}
with $a_i,c_j$ as in Definition~\ref{def:lower-reflection}. As before, the exponent of $t$
sums to $M-1$ and the exponent of $(1-t)$ sums to $-M$, by~\eqref{eq:MEFdef}.

Applying Lemma~\ref{lem:reverseholder}, we obtain
\[
B(M,N) \ge \prod_{i\in\mathcal I} \left( \int_0^1 t^{a_i-1}(1-t)^{-a_i} \, dt \right)^{1/p_i}
         \prod_{j\in\mathcal J} \left( \int_0^1 t^{c_j-1}(1-t)^{-c_j} \, dt \right)^{1/q_j}.
\]
As in the proof of Theorem~\ref{thm:upper-main}, every one of these integrals equals
$\pi/\sin(\pi x_i)$ (resp. $\pi/\sin(\pi x_j)$), regardless of the reflection choice. Hence,
\[
B(M,N) \ge \prod_{i\in\mathcal I} \left( \frac{\pi}{\sin(\pi x_i)} \right)^{\alpha_i/\alpha}
         \prod_{j\in\mathcal J} \left( \frac{\pi}{\sin(\pi x_j)} \right)^{-\alpha_j/\alpha}.
\]
Raising both sides to the power $\alpha>0$ and using $B(M,N)=\pi/\sin(\pi M)$ gives
\[
\left(\frac{\pi}{\sin(\pi M)}\right)^{\alpha} \ge \prod_{i\in\mathcal I} \left( \frac{\pi}{\sin(\pi x_i)} \right)^{\alpha_i}
         \prod_{j\in\mathcal J} \left( \frac{\pi}{\sin(\pi x_j)} \right)^{-\alpha_j}.
\]
This reduces to~\eqref{eq:genineq} since $\displaystyle\prod_{i\in\mathcal I}\pi^{\alpha_i}\prod_{j\in\mathcal J}\pi^{-\alpha_j}=\pi^\alpha$.
The proof is complete.
\end{proof}

\begin{remark}\label{rmk:counting}
In both constructions, replacing a reflection set (resp. pair) by its complement, i.e. 
$E$ by $\mathcal N\setminus E$ (resp. $(E,F)$ by $(\mathcal I\setminus E,\ \mathcal J\setminus F)$), 
sends $M$ to $1-M$, giving the same sine value since $\sin(\pi M)=\sin(\pi(1-M))$. Hence the
$2^n$ reflection choices pair up into $2^{n-1}$ complementary classes, each yielding a single
distinct bound. 

For example, for $n=2$, we obtain two distinct bounds; for $n=3$, we get $2^{3-1}=4$ bounds.
\end{remark}


\section{Comparing the bounds}
\label{sec:compare}

In this section, we consider restrictions of Theorems~\ref{thm:upper-main} and~\ref{thm:lower-main} to the cases $n=2$ and
$n=3$, and determine exactly which member of the resulting family of bounds is the sharpest.

\subsection{Two-dimensional upper bounds}

Let $\alpha_1 = \beta$, $\alpha_2 = \gamma$, $x_1 = x$, $x_2 = y$, and $\alpha = \beta + \gamma > 0$. By Remark~\ref{rmk:counting}, the $2^2=4$ reflection sets collapse to $2$ distinct bounds, attained at $E=\emptyset$ and $E=\{2\}$. Put
\[
U = \frac{\beta x + \gamma y}{\beta + \gamma}, \qquad V = \frac{\beta x + \gamma(1-y)}{\beta + \gamma}.
\]

\begin{corollary}\label{cor:2d-upper}
Let $\beta,\gamma > 0$ and $x,y \in (0,1)$. Assume $U,V \in (0,1)$. Then
\begin{equation}\label{eq:2d-upper}
\sin^{\beta}(\pi x) \sin^{\gamma}(\pi y) \leq \min\left\{\sin^{\beta+\gamma}(\pi U), \sin^{\beta+\gamma}(\pi V)\right\}.
\end{equation}
Moreover,
\[
\sin(\pi U) \geq \sin(\pi V) \iff (1 - 2y)(2x - 1) \geq 0.
\]
\end{corollary}

\begin{proof}
The inequality (\ref{eq:2d-upper}) follows from Theorem~\ref{thm:upper-main}. For the comparison, we note that
\[
U - V = \frac{\gamma(2y - 1)}{\beta + \gamma}, \qquad\text{and}\quad
U + V - 1 = \frac{\beta(2x - 1)}{\beta + \gamma}.
\]
Hence, $(U - V)(U + V - 1) = \frac{\beta\gamma}{(\beta + \gamma)^2}(2y - 1)(2x - 1)$. The result then follows from Lemma~\ref{lem:sine-compare}.
\end{proof}

\begin{corollary}\label{cor:2d-upper-best}
Under the hypotheses of Corollary~\ref{cor:2d-upper}, the best upper bound is
\[
\min\{\sin^{\beta+\gamma}(\pi U), \sin^{\beta+\gamma}(\pi V)\}
=
\begin{cases}
\sin^{\beta+\gamma}(\pi U), & \text{if } (1 - 2y)(2x - 1) \leq 0, \\
\sin^{\beta+\gamma}(\pi V), & \text{if } (1 - 2y)(2x - 1) \geq 0.
\end{cases}
\]
Equality occurs when $y = \frac{1}{2}$ ($U=V$) or $x = \frac{1}{2}$ ($U=1-V$), both giving $\sin(\pi U)=\sin(\pi V)$.
\end{corollary}

\subsection{Two-dimensional lower bounds}

Let $k=1$, $\alpha_1 = \beta$, $\alpha_2 = \alpha$ (with $\beta > \alpha > 0$),
$x_1 = x$, $x_2 = y$, so $\mathcal I=\{1\}$, $\mathcal J=\{2\}$, and the excess
weight is $\beta-\alpha>0$. The $2^2=4$ reflection pairs reduce again to $2$ distinct
bounds, attained at $(E,F)=(\emptyset,\emptyset)$ and $(E,F)=(\emptyset,\{2\})$. Put
\[
U = \frac{\beta x - \alpha y}{\beta-\alpha}, \qquad
V = \frac{\beta x - \alpha(1-y)}{\beta-\alpha}.
\]

\begin{corollary}\label{cor:2d-lower}
For $\beta>\alpha>0$ and $x,y\in(0,1)$ with $U,V\in(0,1)$,
\begin{equation}\label{eq:2d-lower}
\max\left\{
  \sin^{\beta-\alpha}(\pi U),\;
  \sin^{\beta-\alpha}(\pi V)
\right\}
\leq \frac{\sin^\beta(\pi x)}{\sin^\alpha(\pi y)}.
\end{equation}
Moreover,
\[
\sin(\pi U) \ge \sin(\pi V) \;\Longleftrightarrow\; (1-2y)(2x-1) \le 0.
\]
\end{corollary}

\begin{proof}
The inequality (\ref{eq:2d-lower}) follows from applying Theorem~\ref{thm:lower-main} to the two reflection pairs. For the comparison, we obtain
\[
U - V = \frac{\alpha(1-2y)}{\beta-\alpha}, \qquad
U + V - 1 = \frac{\beta(2x-1)}{\beta-\alpha}.
\]
Hence, $(U-V)(U+V-1)=\frac{\alpha\beta}{\beta-\alpha}(1-2y)(2x-1)$. Thus the result follows from Lemma~\ref{lem:sine-compare}.
\end{proof}

\begin{corollary}\label{cor:2d-lower-best}
Under the hypotheses of Corollary~\ref{cor:2d-lower}, the best lower bound is
\[
\max\{\sin^{\beta-\alpha}(\pi U),\sin^{\beta-\alpha}(\pi V)\}
=
\begin{cases}
\sin^{\beta-\alpha}(\pi U), & \text{if } (1-2y)(2x-1) \le 0,\\
\sin^{\beta-\alpha}(\pi V), & \text{if } (1-2y)(2x-1) \ge 0.
\end{cases}
\]
Equality occurs when $y = \frac{1}{2}$ or $x = \frac{1}{2}$.
\end{corollary}

\subsection{Three-dimensional upper bounds}

Let $\alpha_1 = \beta$, $\alpha_2 = \gamma$, $\alpha_3 = \delta$, and $x_1 = x$, $x_2 = y$, $x_3 = z$, so $\alpha = \beta + \gamma + \delta > 0$. By Remark~\ref{rmk:counting}, the $2^3=8$ reflection sets reduce to $2^{3-1}=4$ distinct bounds, attained at $E=\emptyset,\{3\},\{2\},\{1\}$. Define
\[
U_{00} = \frac{\beta x + \gamma y + \delta z}{\alpha},\quad
U_{01} = \frac{\beta x + \gamma y + \delta(1-z)}{\alpha},
\]
\[
U_{10} = \frac{\beta x + \gamma(1-y) + \delta z}{\alpha},\quad
U_{11} = \frac{\beta(1-x) + \gamma y + \delta z}{\alpha}.
\]

\begin{corollary}\label{cor:3d-upper}
Under the above hypotheses with all $U_{ab} \in (0,1)$,
\[
\sin^{\beta}(\pi x) \sin^{\gamma}(\pi y) \sin^{\delta}(\pi z) \leq \min\left\{\sin^{\alpha}(\pi U_{00}), \sin^{\alpha}(\pi U_{01}), \sin^{\alpha}(\pi U_{10}), \sin^{\alpha}(\pi U_{11})\right\}.
\]
\end{corollary}

\begin{proof}
Apply Theorem~\ref{thm:upper-main} to the four reflection sets; taking the minimum of the right-hand sides preserves the inequality.
\end{proof}

\begin{lemma}\label{lem:3d-upper-compare}
With the means defined above, we have the following pairwise comparisons.
\begin{align}
\sin(\pi U_{00}) &\geq \sin(\pi U_{01}) \iff (2z - 1)\left(\beta(2x-1) + \gamma(2y-1)\right) \leq 0, \tag{i} \\
\sin(\pi U_{00}) &\geq \sin(\pi U_{10}) \iff (2y - 1)\left(\beta(2x-1) + \delta(2z-1)\right) \leq 0, \tag{ii} \\
\sin(\pi U_{00}) &\geq \sin(\pi U_{11}) \iff (2x - 1)\left(\gamma(2y-1) + \delta(2z-1)\right) \leq 0, \tag{iii} \\
\sin(\pi U_{01}) &\geq \sin(\pi U_{11}) \iff (2y - 1)\left(\beta(2x-1) + \delta(1-2z)\right) \leq 0, \tag{iv} \\
\sin(\pi U_{10}) &\geq \sin(\pi U_{11}) \iff (2z - 1)\left(\beta(2x-1) + \gamma(1-2y)\right) \leq 0, \tag{v} \\
\sin(\pi U_{01}) &\geq \sin(\pi U_{10}) \iff (2x - 1)\left(\gamma(2y-1) + \delta(1-2z)\right) \leq 0. \tag{vi}
\end{align}
\end{lemma}

\begin{proof}
Compute $U_{ab}-U_{cd}$ and $U_{ab}+U_{cd}-1$ for each pair of means, then apply Lemma~\ref{lem:sine-compare}. For example, for (i):
\[
U_{00} - U_{01} = \frac{\delta(2z-1)}{\alpha}, \qquad
U_{00} + U_{01} - 1 = \frac{\beta(2x-1) + \gamma(2y-1)}{\alpha}.
\]
The other cases are analogous.
\end{proof}

For clarity, we define the following propositional variables:
\begin{align*}
P &: (2z-1)(\beta(2x-1) + \gamma(2y-1)) \leq 0, \\
Q &: (2y-1)(\beta(2x-1) + \delta(2z-1)) \leq 0, \\
R &: (2x-1)(\gamma(2y-1) + \delta(2z-1)) \leq 0, \\
S &: (2x-1)(\gamma(2y-1) + \delta(1-2z)) \leq 0, \\
T &: (2y-1)(\beta(2x-1) + \delta(1-2z)) \leq 0, \\
V &: (2z-1)(\beta(2x-1) + \gamma(1-2y)) \leq 0.
\end{align*}

\begin{theorem}\label{thm:3d-upper-best}
The best upper bound is given by:
\[
\min\{\sin^{\alpha}(\pi U_{00}), \sin^{\alpha}(\pi U_{01}), \sin^{\alpha}(\pi U_{10}), \sin^{\alpha}(\pi U_{11})\}
\]
\[
=
\begin{cases}
\sin^{\alpha}(\pi U_{00}), & \text{if } \neg P,\ \neg Q,\ \neg R \text{ hold},\\
\sin^{\alpha}(\pi U_{01}), & \text{if } P,\ \neg S,\ \neg T \text{ hold},\\
\sin^{\alpha}(\pi U_{10}), & \text{if }  Q,\ S,\ \neg V \text{ hold},\\
\sin^{\alpha}(\pi U_{11}), & \text{if } V,\ T,\ R \text{ hold}.
\end{cases}
\]
Here $\neg P$ denotes the negation of $P$; similarly for the others.
\end{theorem}

\begin{proof}
For each candidate to be the minimum, it must be dominated by the other three bounds. From Lemma~\ref{lem:3d-upper-compare}, we see that

\begin{itemize}
\item[(i)] $\sin(\pi U_{00})$ is the minimum if and only if $\neg P$, $\neg Q$, $\neg R$ hold;
\item[(ii)] $U_{01}$ is the minimum if and only if $P$, $\neg S$, $\neg T$ hold;
\item[(iii)] $U_{10}$ is the minimum if and only if $ Q$, $\neg V$, $S$ hold;
\item[(iv)] $U_{11}$ is the minimum if and only if $V$, $T$, $R$ hold.
\end{itemize}

These four cases are mutually exclusive and exhaustive.
\end{proof}

\subsection{Three-dimensional lower bounds}

We present two distinct signatures for three-dimensional lower bounds: $(+,-,-)$ and $(+,+,-)$.

\subsubsection{Signature $(+,-,-)$}

Let $k=1$, $\alpha_1=\beta$, $\alpha_2=\gamma$, $\alpha_3=\delta$ with $\beta > \gamma + \delta > 0$, so $\alpha = \beta - \gamma - \delta > 0$, $\mathcal I=\{1\}$, $\mathcal J=\{2,3\}$, $x_1=x$, $x_2=y$, $x_3=z$. The four reflection pairs giving distinct bounds are $(E,F)=(\emptyset,\emptyset),(\emptyset,\{2\}),(\emptyset,\{3\}),(\{1\},\emptyset)$. Define
\[
U_{00} = \frac{\beta x - \gamma y - \delta z}{\alpha},\quad
U_{01} = \frac{\beta x - \gamma(1-y) - \delta z}{\alpha},
\]
\[
U_{10} = \frac{\beta x - \gamma y - \delta(1-z)}{\alpha},\quad
U_{11} = \frac{\beta (1-x) - \gamma y - \delta z}{\alpha}.
\]
We deduce the following from the above hypotheses and Theorem \ref{thm:lower-main}.
\begin{corollary}\label{cor:3d-lower-mminus}
Under the above hypotheses with all $U_{ab} \in (0,1)$,
\[
\max\bigl\{
  \sin^{\alpha}(\pi U_{00}),\;
  \sin^{\alpha}(\pi U_{01}),\;
  \sin^{\alpha}(\pi U_{10}),\;
  \sin^{\alpha}(\pi U_{11})
\bigr\}
\;\leq\; \frac{\sin^\beta(\pi x)}{\sin^\gamma(\pi y)\sin^\delta(\pi z)}.
\]
\end{corollary}

\begin{lemma}\label{lem:3d-lower-mminus-compare}
With the means defined above, we have:
\begin{align}
\sin(\pi U_{00}) \ge \sin(\pi U_{01}) &\iff (1-2y)\bigl(\beta(2x-1) - \delta(2z-1)\bigr) \le 0, \tag{i}\\
\sin(\pi U_{00}) \ge \sin(\pi U_{10}) &\iff (1-2z)\bigl(\beta(2x-1) - \gamma(2y-1)\bigr) \le 0, \tag{ii}\\
\sin(\pi U_{00}) \ge \sin(\pi U_{11}) &\iff (2x-1)\bigl(\gamma(1-2y) + \delta(1-2z)\bigr) \le 0, \tag{iii}\\
\sin(\pi U_{01}) \ge \sin(\pi U_{11}) &\iff (1-2z)\bigl(\beta(2x-1) + \gamma(2y-1)\bigr) \le 0, \tag{iv}\\
\sin(\pi U_{10}) \ge \sin(\pi U_{11}) &\iff (1-2y)\bigl(\beta(2x-1) + \delta(2z-1)\bigr) \le 0, \tag{v}\\
\sin(\pi U_{01}) \ge \sin(\pi U_{10}) &\iff (2x-1)\bigl(\gamma(2y-1) + \delta(1-2z)\bigr) \le 0. \tag{vi}
\end{align}
\end{lemma}

The best lower bound is determined by the following theorem and its proof is analogous to Theorem~\ref{thm:3d-upper-best}.

\begin{theorem}\label{thm:3d-lower-mminus-best}
The best lower bound is given by
\[
\max\{\sin^{\alpha}(\pi U_{00}), \sin^{\alpha}(\pi U_{01}), \sin^{\alpha}(\pi U_{10}), \sin^{\alpha}(\pi U_{11})\}
\]
\[
=
\begin{cases}
\sin^{\alpha}(\pi U_{00}), & \text{if } P,\ Q,\ T \text{ hold},\\
\sin^{\alpha}(\pi U_{01}), & \text{if } \neg P,\ R,\ V \text{ hold},\\
\sin^{\alpha}(\pi U_{10}), & \text{if } \neg Q,\ S,\ \neg V \text{ hold},\\
\sin^{\alpha}(\pi U_{11}), & \text{if } \neg R,\ \neg S,\ \neg T \text{ hold},
\end{cases}
\]
where
\begin{align*}
P &: (1-2y)\bigl(\beta(2x-1) - \delta(2z-1)\bigr) \le 0, \\
Q &: (1-2z)\bigl(\beta(2x-1) - \gamma(2y-1)\bigr) \le 0, \\
R &: (1-2z)\bigl(\beta(2x-1) + \gamma(2y-1)\bigr) \le 0, \\
S &: (1-2y)\bigl(\beta(2x-1) + \delta(2z-1)\bigr) \le 0, \\
T &: (2x-1)\bigl(\gamma(1-2y) + \delta(1-2z)\bigr) \le 0, \\
V &: (2x-1)\bigl(\gamma(2y-1) + \delta(1-2z)\bigr) \le 0.
\end{align*}
\end{theorem}

\subsubsection{Signature $(+,+,-)$}

Let $k=2$, $\alpha_1=\beta$, $\alpha_2=\gamma$, $\alpha_3=\delta$ with $\beta+\gamma>\delta>0$, so $\alpha = \beta + \gamma - \delta > 0$, $\mathcal I=\{1,2\}$, $\mathcal J=\{3\}$, $x_1=x$, $x_2=y$, $x_3=z$. The four reflection pairs giving distinct bounds are $(E,F)=(\emptyset,\emptyset),(\{2\},\emptyset),(\{1\},\emptyset),(\{1,2\},\emptyset)$. Define
\[
U_{00} = \frac{\beta x + \gamma y - \delta z}{\alpha},\quad
U_{01} = \frac{\beta x + \gamma(1-y) - \delta z}{\alpha},
\]
\[
U_{10} = \frac{\beta(1-x) + \gamma y - \delta z}{\alpha},\quad
U_{11} = \frac{\beta(1-x) + \gamma(1-y) - \delta z}{\alpha}.
\]
We deduce the following from Theorem \ref{thm:lower-main}.
\begin{corollary}\label{cor:3d-lower-ppminus}
Under the above hypotheses with all $U_{ab} \in (0,1)$,
\[
\max\bigl\{
  \sin^{\alpha}(\pi U_{00}),\;
  \sin^{\alpha}(\pi U_{01}),\;
  \sin^{\alpha}(\pi U_{10}),\;
  \sin^{\alpha}(\pi U_{11})
\bigr\}
\;\leq\; \frac{\sin^\beta(\pi x)\sin^\gamma(\pi y)}{\sin^\delta(\pi z)}.
\]
\end{corollary}

\begin{lemma}\label{lem:3d-lower-ppminus-compare}
With the means defined above, we have
\begin{align}
\sin(\pi U_{00}) \ge \sin(\pi U_{01}) &\iff (2y-1)\bigl(\beta(2x-1) + \delta(1-2z)\bigr) \le 0, \tag{i}\\
\sin(\pi U_{00}) \ge \sin(\pi U_{10}) &\iff (2x-1)\bigl(\gamma(2y-1) + \delta(1-2z)\bigr) \le 0, \tag{ii}\\
\sin(\pi U_{00}) \ge \sin(\pi U_{11}) &\iff \bigl(\beta(2x-1) + \gamma(2y-1)\bigr)(1-2z) \le 0, \tag{iii}\\
\sin(\pi U_{01}) \ge \sin(\pi U_{11}) &\iff (2x-1)\bigl(\gamma(1-2y) + \delta(1-2z)\bigr) \le 0, \tag{iv}\\
\sin(\pi U_{10}) \ge \sin(\pi U_{11}) &\iff (2y-1)\bigl(\beta(1-2x) + \delta(1-2z)\bigr) \le 0, \tag{v}\\
\sin(\pi U_{01}) \ge \sin(\pi U_{10}) &\iff (1-2z)\bigl(\beta(2x-1) + \gamma(1-2y)\bigr) \le 0. \tag{vi}
\end{align}
\end{lemma}

The best lower bound is given in the following. The proof is handled as above.

\begin{theorem}\label{thm:3d-lower-ppminus-best}
The best lower bound is given by
\[
\max\{\sin^{\alpha}(\pi U_{00}), \sin^{\alpha}(\pi U_{01}), \sin^{\alpha}(\pi U_{10}), \sin^{\alpha}(\pi U_{11})\}
\]
\[
=
\begin{cases}
\sin^{\alpha}(\pi U_{00}), & \text{if } P,\ Q,\ R \text{ hold},\\
\sin^{\alpha}(\pi U_{01}), & \text{if } \neg P,\ S,\ V \text{ hold},\\
\sin^{\alpha}(\pi U_{10}), & \text{if } \neg Q,\ T,\ \neg V \text{ hold},\\
\sin^{\alpha}(\pi U_{11}), & \text{if } \neg R,\ \neg S,\ \neg T \text{ hold},
\end{cases}
\]
where
\begin{align*}
P &: (2y-1)\bigl(\beta(2x-1) + \delta(1-2z)\bigr) \le 0, \\
Q &: (2x-1)\bigl(\gamma(2y-1) + \delta(1-2z)\bigr) \le 0, \\
R &: \bigl(\beta(2x-1) + \gamma(2y-1)\bigr)(1-2z) \le 0, \\
S &: (2x-1)\bigl(\gamma(1-2y) + \delta(1-2z)\bigr) \le 0, \\
T &: (2y-1)\bigl(\beta(1-2x) + \delta(1-2z)\bigr) \le 0, \\
V &: (1-2z)\bigl(\beta(2x-1) + \gamma(1-2y)\bigr) \le 0.
\end{align*}
\end{theorem}


\section{Concrete examples}
\label{sec:examples}

In this section, we provide detailed numerical examples for each case, presenting the computations in comprehensive tables that allow easy verification of the results. For each case, we display multiple parameter choices to illustrate the different dominance regions.

\subsection{Two-dimensional upper bound example}

Choose $\beta = x + y$, $\gamma = y$, so $\alpha = x + 2y > 0$. Then
\[
U = \frac{(x+y)x + y^2}{x+2y}, \qquad V = \frac{(x+y)x + y(1-y)}{x+2y}.
\]

\begin{proposition}\label{prop:2d-upper-concrete}
For $x,y \in (0,1)$ with $0 < U,V < 1$,
\[
\sin^{x+y}(\pi x) \sin^y(\pi y) \leq \min\{\sin^{x+2y}(\pi U), \sin^{x+2y}(\pi V)\}.
\]
Moreover,
\[
\min\{\sin^{x+2y}(\pi U), \sin^{x+2y}(\pi V)\}
=
\begin{cases}
\sin^{x+2y}(\pi V), & \text{if } (1-2y)(2x-1) \leq 0,\\
\sin^{x+2y}(\pi U), & \text{if } (1-2y)(2x-1) \geq 0.
\end{cases}
\]
\end{proposition}

\begin{example}[Numerical illustrations]
The following table shows results for various $(x,y)$ pairs. Recall that for upper bounds, we take the minimum of the two bounds.

\begin{table}[h]
\centering
\begin{tabular}{|c|c|c|c|c|c|}
\hline
$(x,y)$ & $U$ & $V$ & $\sin^{x+2y}(\pi U)$ & $\sin^{x+2y}(\pi V)$ & Best Upper Bound \\
\hline
$(0.3,0.7)$ & 0.4647 & 0.3000 & 0.990 & 0.697 & $ 0.697$ \\
$(0.4,0.6)$ & 0.4750 & 0.4000 & 0.995 & 0.923 & $ 0.923$ \\
$(0.5,0.3)$ & 0.4455 & 0.5545 & 0.984 & 0.984 & $ 0.984$ \\
$(0.6,0.4)$ & 0.5429 & 0.6000 & 0.987 & 0.932 & $ 0.932$ \\
$(0.7,0.2)$ & 0.6091 & 0.7182 & 0.936 & 0.755 & $ 0.755$ \\
\hline
\end{tabular}
\end{table}
\noindent The best upper bound is the minimum of the two bounds in each row.
\end{example}

\subsection{Two-Dimensional Lower Bound Example}

Choose $\beta = x + y$ and $\alpha = y$, so $\beta-\alpha = x > 0$. Then
\[
U = \frac{x^2 + xy - y^2}{x}, \qquad
V = \frac{x^2 + xy - y + y^2}{x}.
\]

\begin{proposition}\label{prop:2d-lower-concrete}
For $x,y\in(0,1)$ with $0<U,V<1$,
\[
\max\{\sin^{x}(\pi U),\;\sin^{x}(\pi V)\} \;\leq\; \frac{\sin^{x+y}(\pi x)}{\sin^{y}(\pi y)}.
\]
Moreover,
\[
\max\{\sin^{x}(\pi U),\;\sin^{x}(\pi V)\}
=
\begin{cases}
\sin^{x}(\pi U), & \text{if } (1-2y)(2x-1) \le 0,\\
\sin^{x}(\pi V), & \text{if } (1-2y)(2x-1) \ge 0.
\end{cases}
\]
\end{proposition}

\begin{example}[Numerical illustrations]
The following table shows results for various $(x,y)$ pairs. Recall that for lower bounds, we take the maximum of the two bounds. The last row requires $x=0.5$, which gives $U=1-V$ automatically by the equality case of Corollary~\ref{cor:2d-lower}.

\begin{table}[h]
\centering
\begin{tabular}{|c|c|c|c|c|c|}
\hline
$(x,y)$ & $U$ & $V$ & $\sin^{x}(\pi U)$ & $\sin^{x}(\pi V)$ & Best Lower Bound \\
\hline
$(0.6,0.3)$ & 0.7500 & 0.5500 & 0.812 & 0.993 & $ 0.993$ \\
$(0.7,0.4)$ & 0.8714 & 0.7571 & 0.520 & 0.772 & $ 0.772$ \\
$(0.5,0.2)$ & 0.6200 & 0.3800 & 0.964 & 0.964 & $ 0.964$ \\
$(0.4,0.3)$ & 0.4750 & 0.1750 & 0.999 & 0.771 & $ 0.999$ \\
$(0.5,0.35)$ & 0.6050 & 0.3950 & 0.973 & 0.973 & $ 0.973$ \\
\hline
\end{tabular}
\end{table}
\noindent The best lower bound is the maximum of the two bounds in each row.
\end{example}

\subsection{Three-dimensional upper bound Example}

Choose $\beta = x$, $\gamma = y$, $\delta = z$, so $\alpha = x + y + z > 0$. Define
\[
U_{00} = \frac{x^2 + y^2 + z^2}{x+y+z}, \quad
U_{01} = \frac{x^2 + y^2 + z(1-z)}{x+y+z},
\]
\[
U_{10} = \frac{x^2 + y(1-y) + z^2}{x+y+z}, \quad
U_{11} = \frac{x(1-x) + y^2 + z^2}{x+y+z}.
\]

Then by Corollary~\ref{cor:3d-upper}, we have the following.
\begin{proposition}\label{prop:3D-prodexamp}
For $x,y,z\in (0,1)$ such that all $U_{ab}\in (0,1)$, we have
\[
\sin^x(\pi x) \sin^y(\pi y) \sin^z(\pi z) \leq \min\{\sin^{x+y+z}(\pi U_{00}), \sin^{x+y+z}(\pi U_{01}), \sin^{x+y+z}(\pi U_{10}), \sin^{x+y+z}(\pi U_{11})\}.
\]
\end{proposition}
\begin{example}[Numerical Illustrations]
The following table shows results for various $(x,y,z)$ triples. Recall that for upper bounds, we take the minimum of all bounds.

\begin{table}[h]
\centering
\caption{Three-dimensional upper bound results for various $(x,y,z)$}
\label{tab:3d-upper-multi}
\begin{tabular}{|c|c|c|c|c|c|c|}
\hline
$(x,y,z)$ & $\sin^{\alpha}(\pi U_{00})$ & $\sin^{\alpha}(\pi U_{01})$ & $\sin^{\alpha}(\pi U_{10})$ & $\sin^{\alpha}(\pi U_{11})$ & Best Upper Bound \\
\hline
$(0.2,0.4,0.6)$ & 0.993 & 0.897 & 0.993 & 0.974 & $ 0.897$ \\
$(0.3,0.3,0.4)$ & 0.876 & 0.969 & 0.992 & 0.992 & $ 0.876$ \\
$(0.4,0.5,0.1)$ & 0.969 & 1.000 & 0.969 & 1.000 &  $ 0.969$ \\
$(0.5,0.2,0.3)$ & 0.930 & 1.000 & 1.000 & 0.930 &  $ 0.930$ \\
$(0.6,0.3,0.5)$ & 1.000 & 1.000 & 0.950 & 0.950 & $ 0.950$ \\
\hline
\end{tabular}
\end{table}
\noindent The best upper bound is the minimum of all bounds in each row.
\end{example}

\subsection{Three-dimensional lower bound: signature $(+,-,-)$ example}

Choose $\beta = x+y+z$, $\gamma = y$, $\delta = z$, so $\alpha = x > 0$. Define
\[
U_{00} = \frac{x^2+xy+xz-y^2-z^2}{x},\;
U_{01} = \frac{x^2+xy+xz-y+y^2-z^2}{x},
\]
\[
U_{10} = \frac{x^2+xy+xz-y^2-z+z^2}{x},\;
U_{11} = \frac{x^2+xy+xz-y+y^2-z+z^2}{x}.
\]

\begin{proposition}\label{prop:3d-lower-mminus-concrete}
For $x,y,z\in(0,1)$ with all $U_{ab}\in(0,1)$,
\[
\max\{B_U,\;B_V,\;B_W,\;B_Z\}
\;\leq\; \frac{\sin^{x+y+z}(\pi x)}{\sin^{y}(\pi y)\sin^{z}(\pi z)},
\]
where
\[
B_U = \sin^{x}(\pi U_{00}),\quad B_V = \sin^{x}(\pi U_{01}),\quad B_W = \sin^{x}(\pi U_{10}),\quad B_Z = \sin^{x}(\pi U_{11}).
\]
\end{proposition}

\begin{example}[Numerical illustrations]
The following table shows results for various $(x,y,z)$ triples. Recall that for lower bounds, we take the maximum of all bounds. 

\begin{table}[h]
\centering
\caption{Three-dimensional lower bound $(+,-,-)$ results for various $(x,y,z)$}
\label{tab:3d-lower-mminus-multi}
\begin{tabular}{|c|c|c|c|c|c|c|}
\hline
$(x,y,z)$ & $B_U$ & $B_V$ & $B_W$ & $B_Z$ & Best Lower Bound \\
\hline
$(0.5,0.2,0.15)$ & 0.872 & 0.999 & 0.999 & 0.872 & $0.999$ \\
$(0.6,0.3,0.2)$ & 0.540 & 0.900 & 0.900 & 0.999 &  $ 0.999$ \\
$(0.4,0.15,0.1)$ & 0.991 & 0.924 & 0.966 & 0.640 &  $ 0.991$ \\
$(0.6,0.25,0.2)$ & 0.551 & 0.913 & 0.904 & 0.997 &  $ 0.997$ \\
$(0.4,0.1,0.05)$ & 0.999 & 0.934 & 0.983 & 0.817 & $ 0.999$ \\
\hline
\end{tabular}
\end{table}
\noindent The best lower bound is the maximum of all bounds in each row.
\end{example}

\subsection{Three-dimensional lower bound: Signature $(+,+,-)$ Example}

Choose $\beta = x$, $\gamma = y$, $\delta = z$ with $x+y>z$, so $\alpha = x+y-z > 0$. Define
\[
U_{00} = \frac{x^2+y^2-z^2}{x+y-z},\;
U_{01} = \frac{x^2+y(1-y)-z^2}{x+y-z},
\]
\[
U_{10} = \frac{x(1-x)+y^2-z^2}{x+y-z},\;
U_{11} = \frac{x(1-x)+y(1-y)-z^2}{x+y-z}.
\]

\begin{proposition}\label{prop:3d-lower-ppminus-concrete}
For $x,y,z\in(0,1)$ with $x+y>z$ and all $U_{ab}\in(0,1)$,
\[
\max\{C_U,\;C_V,\;C_W,\;C_Z\}
\;\leq\; \frac{\sin^{x}(\pi x)\sin^{y}(\pi y)}{\sin^{z}(\pi z)},
\]
where
\[
C_U = \sin^{x+y-z}(\pi U_{00}),\quad C_V = \sin^{x+y-z}(\pi U_{01}),\quad C_W = \sin^{x+y-z}(\pi U_{10}),\quad C_Z = \sin^{x+y-z}(\pi U_{11}).
\]
\end{proposition}

\begin{example}[Numerical illustrations]
The following table shows results for various $(x,y,z)$ triples. Recall that for lower bounds, we take the maximum of all bounds.

\begin{table}[h]
\centering
\caption{Three-dimensional lower bound $(+,+,-)$ results for various $(x,y,z)$}
\label{tab:3d-lower-ppminus-multi}
\begin{tabular}{|c|c|c|c|c|c|c|}
\hline
$(x,y,z)$ & $C_U$ & $C_V$ & $C_W$ & $C_Z$ & Best Lower Bound \\
\hline
$(0.4,0.5,0.2)$ & 0.997 & 0.997 & 0.930 & 0.930 & $0.997$ \\
$(0.3,0.4,0.15)$ & 0.980 & 0.991 & 0.953 & 0.785  & $0.991$ \\
$(0.5,0.3,0.1)$ & 0.997 & 0.930 & 0.997 & 0.930 &  $ 0.997$ \\
$(0.6,0.4,0.25)$ & 0.955 & 0.828 & 0.991 & 0.988 & $ 0.991$ \\
$(0.2,0.6,0.1)$ & 0.989 & 0.955 & 0.820 & 0.989 & $ 0.989$ \\
\hline
\end{tabular}
\end{table}
\noindent The best lower bound is the maximum of all bounds in each row.
\end{example}


\section{Conclusion}
\label{sec:conclusion}

We have introduced the reflection-substitution technique: the observation that the identity $\sin(\pi x_i) = \sin(\pi(1-x_i))$ furnishes, at each coordinate independently, two equivalent Beta-integral representations, and used it, together with H\"older's inequality and its reverse, to build a unified framework for weighted sine inequalities.

This approach yields $2^{n-1}$ distinct upper bounds for the weighted product $\displaystyle\prod_{i=1}^n\sin^{\alpha_i}(\pi x_i)$, and  $2^{n-1}$ distinct lower bounds for sine ratios, extending the method to a direction not previously treated. In both families, the comparison between bounds reduces to simple sign conditions on linear combinations of $(2x_i-1)$.

 Detailed numerical examples were presented in comprehensive tables with multiple parameter choices illustrating the theory in two- and three-dimensional cases, with the best bound clearly identified for each example.

Extending the method (comparison) to $n \geq 4$ dimensions would require more elaborate comparison criteria, but the general framework developed here already provides the template. The number of bounds grows as $2^{n-1}$, and the comparison problem becomes increasingly complex but structurally similar.

Since weighted sine products and ratios are disguised Beta-function estimates, having a selectable family of bounds rather than a single fixed one is directly useful wherever such estimates arise.
The closed-form dominance criteria mean the best bound can be identified explicitly for any given numerical configuration, rather than merely known to exist.

It would be interesting to study $q$-analogues of our results via for example the $q$-Gamma function as suggested for the upper bound side in~\cite{MSW2026}. The $q$-Beta function has analogous reflection formulas that may permit a similar technique.

\section{Compliance with Ethical Standards}
\begin{itemize}
\item {\bf Funding}

There is no funding support to declare.

\item {\bf Disclosure of potential conflicts of interest}

The authors have no relevant financial or non-financial interests to disclose.

\item {\bf Author Contributions}

All authors contributed to the conception and design of the study.  All authors read and approved the final manuscript.

\item {\bf Data availability statements}

Data sharing is not applicable to this article, as no data was created or analyzed in this study.
\end{itemize}


\begin{thebibliography}{99}

\bibitem{Artin1964}
E.~Artin,
\textit{The Gamma Function},
translated by M.~Butler,
Holt, Rinehart and Winston, New York, 1964.

\bibitem{BohrMollerup1922}
H.~Bohr and J.~Mollerup,
\textit{L{\ae}rebog i matematisk analyse}, Vol.~3,
Jul.~Gjellerups Forlag, Copenhagen, 1922.

\bibitem{Bullen2003}
P.~S.~Bullen,
\textit{Handbook of Means and Their Inequalities},
Springer Dordrecht, 2003.

\bibitem{HLP1934}
G.~H.~Hardy, J.~E.~Littlewood, and G.~P\'{o}lya,
\textit{Inequalities},
Cambridge University Press, Cambridge, 1934.

\bibitem{MSW2026}
A.~L.~Mahu, B.~F.~Sehba, and C.~D.~Williams,
\textit{Weighted Product Inequalities for the Sine Function: A Gamma-Function Approach and Sharp Comparisons},
arXiv:2604.13106v1 [math.GM], 2026.

\bibitem{Mitrinovic1970}
D.~S.~Mitrinovi\'c,
\textit{Analytic Inequalities},
Springer, Berlin, 1970.

\bibitem{MPF1993}
D.~S.~Mitrinovi\'c, J.~E.~Pe\v{c}ari\'c, and A.~M.~Fink,
\textit{Classical and New Inequalities in Analysis},
Springer Dordrecht, 1992.

\bibitem{Rudin1987}
W.~Rudin,
\textit{Real and Complex Analysis}, 3rd~ed.,
McGraw-Hill, New York, 1987.

\end{thebibliography}
\end{document}